\documentclass[11pt, oneside]{amsart}
\usepackage[utf8]{inputenc}

\usepackage{amsmath}%
\usepackage{amsfonts}%
\usepackage{amssymb}%
\usepackage{amsthm}
\usepackage{bm}
\usepackage{graphicx}
\usepackage{stmaryrd}
\usepackage{centernot} 
\usepackage{tikz}
\usepackage{wrapfig}
\usepackage{mathtools}

\usepackage{url}
\usepackage{geometry}
\usepackage{enumerate}
\usepackage{hyperref}

\theoremstyle{plain}
\newtheorem{theorem}{Theorem}[section]

\newtheorem{proposition}[theorem]{Proposition}
\newtheorem{problem}{Problem}

\theoremstyle{definition}
\newtheorem{definition}[theorem]{Definition}

\newtheorem{remark}[theorem]{Remark}

\newcommand{\abs}[1]{\vert #1 \vert}

\newcommand{\set}[1]{\{#1\}} 
\renewcommand{\restriction}{\mathord{\upharpoonright}}

\title{Complexity of codes for Ramsey positive sets}

\author{Allison Wang}
\address{Department of Mathematical Sciences, Carnegie Mellon University}
\email{ayw2@andrew.cmu.edu}

\begin{document}

\begin{abstract}
Sabok showed that the set of codes for $G_\delta$ Ramsey positive subsets of $[\omega]^\omega$ is $\mathbf{\Sigma}^1_2$-complete. We extend this result by providing sufficient conditions for the set of codes for $G_\delta$ Ramsey positive subsets of an arbitrary topological Ramsey space to be $\mathbf{\Sigma}^1_2$-complete.
\end{abstract}

\maketitle

% -------------- INTRODUCTION --------------------------------
\section{Introduction}
A well-known theorem of Ramsey states that given any $k < \omega$ and any $\mathcal{X} \subseteq [\omega]^k$, there is an infinite set $A \subseteq \omega$ such that $[A]^k \subseteq \mathcal{X}$ or $[A]^k \cap \mathcal{X} = \emptyset$. This result fails for $k = \omega$ when the Axiom of Choice is assumed. However, there is a topological characterization of which sets $\mathcal{X} \subseteq [\omega]^\omega$ satisfy a stronger related property. Recall that the {\it Ellentuck topology} on $[\omega]^\omega$ is generated by all sets of the form
$$[s, A] = \set{B \in [\omega]^\omega \mid s \sqsubseteq B \wedge B \subseteq A}$$
where $s \in [\omega]^{< \omega}$, $A \in [\omega]^\omega$, and $s \sqsubseteq B$ means that $s$ is an initial segment of $B$. A set $\mathcal{X} \subseteq [\omega]^\omega$ is {\it Ramsey} if for all non-empty $[s, A]$, there is $B \in [s, A]$ such that $[s, B] \subseteq \mathcal{X}$ or $[s, B] \cap \mathcal{X} = \emptyset$. A set is {\it Ramsey null} if the latter always holds, and {\it Ramsey positive} otherwise. Silver showed that every analytic set is Ramsey \cite{Silver}. In fact, Ellentuck proved a stronger result: a set is Ramsey iff it has the Baire property in the Ellentuck topology \cite{Ellentuck}. We discuss these definitions and analogous results in a more general setting in Section \ref{secBackground}.

Sabok proved that the set of codes for $G_\delta$ (i.e., $\mathbf{\Pi}^0_2$) Ramsey positive sets in $[\omega]^\omega$ is $\mathbf{\Sigma}^1_2$-complete \cite[Theorem 1]{Sabok}; see Section \ref{secBackground} for the relevant definitions. In the same paper, Sabok proved a general theorem about $\mathbf{\Sigma}^1_2$-complete sets that has been used in several recent proofs of $\mathbf{\Sigma}^1_2$-completeness \cite[Theorem 2]{Sabok}. For instance, Todor\v{c}evi\'{c} and Vidny\'{a}nszky showed that the set of closed subgraphs of the shift graph on $[\omega]^\omega$ that have finite Borel chromatic number is $\mathbf{\Sigma}^1_2$-complete \cite[Theorem 1.3]{TodorcevicVidnyanszky}. Brandt, Chang, Greb\'{i}k, Grunau, Rozho\v{n}, and Vidny\'{a}nszky showed a similar result for graphs of bounded degree: for any $d > 2$, the set of Borel acyclic $d$-regular graphs with Borel chromatic number at most $d$ is $\mathbf{\Sigma}^1_2$-complete \cite[Theorem 1.2]{onHomGraphs}. Finally, Thornton proved that a more general class of Borel constraint satisfaction problems is $\mathbf{\Sigma}^1_2$-complete, including several other examples from Borel combinatorics \cite[Theorem 1.7]{Thornton}.

To prove \cite[Theorem 1]{Sabok}, Sabok considered the set of codes corresponding to a universal $G_\delta$ set formed by viewing every $x \in 2^\omega$ as representing a countable sequence of trees. We show that we can find a continuous reduction from this set of codes to the analogous set of codes associated with a topological Ramsey space satisfying axioms specified by Todor\v{c}evi\'{c} in \cite{Ramsey} whenever the space is sufficiently similar to $[\omega]^\omega$. We call such spaces {\it well-indexed}; see Definition \ref{properties}.
\begin{theorem}\label{mainThm}
    Suppose $(\mathcal{R}, \leq, r)$ satisfies A.1-A.4 from \cite[Section 5.1]{Ramsey}, $\mathcal{R}$ is closed, $\mathcal{AR}$ is countable, and $(\mathcal{R}, \leq, r)$ is well-indexed. Then the set of codes for $G_\delta$ Ramsey positive subsets of $\mathcal{R}$ is $\mathbf{\Sigma}^1_2$-complete.
\end{theorem}

It turns out that many topological Ramsey spaces satisfy the conditions of this result. We present a few examples in Section \ref{secExamples}.

% --------------- ACKNOWLEDGMENTS ------------------------------------
\section{Acknowledgments}
Thank you to Clinton Conley, Natasha Dobrinen, and Aristotelis Panagiotopoulos for helpful discussions, and to the referee for many useful suggestions. The author was supported by the National Science Foundation Graduate Research Fellowship Program under Grant No.~DGE2140739, and by the ARCS Foundation.

% ----------------- BACKGROUND ------------------------------------
\section{Background}\label{secBackground}
Throughout, we will let $(\mathcal{R}, \leq, r)$ denote a triple where $\mathcal{R}$ is a non-empty set, $\leq$ is a reflexive and transitive relation on $\mathcal{R}$, and $r: \mathcal{R} \times \omega \to \mathcal{AR}$ is a map into a set of finite approximations of $\mathcal{R}$. We assume that $\mathcal{AR}$ is countable. In addition, we assume that $(\mathcal{R}, \leq, r)$ satisfies axioms A.1, A.2, A.3, and A.4 from \cite[Section 5.1]{Ramsey} and that $\mathcal{R}$ is closed as a subspace of $\mathcal{AR}^\omega$, where $\mathcal{AR}$ is given the discrete topology. Unless otherwise stated, we use $s, t, \ldots$ to denote elements of $\mathcal{AR}$ and $A, B, \ldots$ to denote elements of $\mathcal{R}$. We write $\mathcal{R}$ for the triple $(\mathcal{R}, \leq, r)$ when clear from context.

We write $r_n(A)$ to mean $r(A, n)$ for $A \in \mathcal{R}$ and $n < \omega$. For $s \in \mathcal{AR}$ and $A \in \mathcal{R}$, define $s \sqsubseteq A$ iff $s = r_n(A)$ for some $n < \omega$. For $s, t \in \mathcal{AR}$, define $s \sqsubseteq t$ iff there are $A \in \mathcal{R}$ and $m \leq n < \omega$ such that $s = r_m(A)$ and $t = r_n(A)$.

We make use of two different topologies on $\mathcal{R}$. We call the topology induced as a subspace of $\mathcal{AR}^\omega$ the {\it metrizable topology}. Unless otherwise specified, all topological notions are taken to be in the metrizable topology. Note that the metrizable topology on $\mathcal{R}$ is Polish since $\mathcal{AR}$ is countable and $\mathcal{R}$ is closed. Let
$$[s, A] := \set{B \in \mathcal{R} \mid s \sqsubseteq B \wedge B \leq A}.$$
We call the topology generated by all $[s, A]$ the {\it Ellentuck topology}.

A set $\mathcal{X} \subseteq \mathcal{R}$ is {\it Ramsey} if for every $[s, A] \neq \emptyset$, there is $B \in [s, A]$ such that $[s, B] \subseteq \mathcal{X}$ or $[s, B] \cap \mathcal{X} = \emptyset$. We say that $\mathcal{X}$ is {\it Ramsey null} if for every $[s, A] \neq \emptyset$, there is $B \in [s, A]$ such that $[s, B] \cap \mathcal{X} = \emptyset$. If $\mathcal{X}$ is not Ramsey null, then we say that $\mathcal{X}$ is {\it Ramsey positive}. Note that a Ramsey positive set need not be Ramsey in general.

Todor\v{c}evi\'{c} proved the following general connection between the Ellentuck topology and these Ramsey-theoretic notions \cite[Theorem 5.4]{Ramsey}:
\begin{theorem}[Todor\v{c}evi\'{c}]
    If $(\mathcal{R}, \leq, r)$ satisfies axioms A.1-A.4 from \cite[Section 5.1]{Ramsey} and $\mathcal{R}$ is closed as a subset of $\mathcal{AR}^\omega$, then every subset of $\mathcal{R}$ with the Baire property in the Ellentuck topology is Ramsey, and every Ellentuck meager set is Ramsey null.
\end{theorem}

Every analytic set has the Baire property in the Ellentuck topology and is therefore Ramsey \cite[Corollary 11]{Ellentuck}. So for every analytic $\mathcal{X} \subseteq \mathcal{R}$, we have that $\mathcal{X}$ is Ramsey positive iff $\mathcal{X}$ contains some $[s, A] \neq \emptyset$.

Recall that $([\omega]^\omega, \subseteq, r)$ satisfies the assumptions above, where $r_n$ is the function mapping an element of $[\omega]^\omega$ to the set of its least $n$ elements. We take $s', t', \ldots$ to denote elements of $[\omega]^{< \omega}$ and $A', B', \ldots$ to denote elements of $[\omega]^\omega$ unless stated otherwise.

Let $\mathbf{\Gamma}$ be a pointclass, $\mathcal{A} \subseteq \mathbf{\Gamma}(X)$, and $U \subseteq \omega^\omega \times X$ a universal $\mathbf{\Gamma}$ set. We call $\set{x \in \omega^\omega \mid U_x \in \mathcal{A}}$ the {\it set of codes for $\mathcal{A}$ in $X$}. We will primarily be interested in the set of codes for $G_\delta$ Ramsey positive sets in a topological Ramsey space $\mathcal{R}$ as above.

% -------------- COMPLEXITY CALCULATION ---------------------------------
\section{The set of codes for \texorpdfstring{$G_\delta$}{G delta} Ramsey positive sets is \texorpdfstring{$\mathbf{\Sigma}^1_2$}{Sigma 1 2}}\label{secCalc}

Following \cite{Sabok}, we define a universal $G_\delta$ set $G \subseteq 2^\omega \times \mathcal{R}$ by interpreting each $x \in 2^\omega$ as a code for a sequence of closed sets. Since $\mathcal{AR}$ is countable, we can view each $x \in 2^\omega$ as coding a sequence $\langle X_n \subseteq \mathcal{AR} \mid n < \omega \rangle$ in a uniform way. If every $X_n$ is a tree with respect to $\sqsubseteq$, then let
\begin{equation*}
    G_x = \mathcal{R} \setminus \bigcup_{n < \omega} [X_n],
\end{equation*}
where $[X_n] = \set{A \in \mathcal{R} \mid \forall k < \omega~r_k(A) \in X_n}$. Otherwise, let $G_x = \emptyset$. Then every $G_\delta$ subset of $\mathcal{R}$ is realized as a section of $G$. Define
\begin{equation*}
    C := \set{x \in 2^\omega \mid G_x \text{ is Ramsey positive}}.
\end{equation*}

\begin{proposition}\label{sigma12}
    The set $C$ is $\mathbf{\Sigma}^1_2.$
\end{proposition}

\begin{proof}
Observe that for $x \in 2^\omega$ coding $\langle X_n \subseteq \mathcal{AR} \mid n < \omega \rangle$, we have $x \in C$ iff
\begin{enumerate}
    \item for all $n < \omega$, $X_n$ is a tree;
    \item $\mathcal{R} \setminus \bigcup_{n < \omega} [X_n]$ is Ramsey positive.
\end{enumerate}
Note that condition (1) is closed, so it suffices to verify that condition (2) is $\mathbf{\Sigma}^1_2$. Since $\mathcal{R} \setminus \bigcup_{n < \omega} [X_n]$ is analytic, condition (2) holds iff $\mathcal{R} \setminus \bigcup_{n < \omega} [X_n]$ contains some non-empty $[s, A]$; equivalently,
\begin{equation*}
    \exists s \in \mathcal{AR}~\exists A \in \mathcal{R} \big(s \sqsubseteq A \wedge \forall n < \omega~[s, A] \cap [X_n] = \emptyset \big).
\end{equation*}
Since $\mathcal{AR}$ is countable and $\mathcal{R}$ is a Polish space, this condition is $\mathbf{\Sigma}^1_2$. We conclude that $C$ is $\mathbf{\Sigma}^1_2$.
\end{proof}

% -------------------------- SIGMA^1_2 HARD -------------------------
\section{Conditions for \texorpdfstring{$\mathbf{\Sigma}^1_2$}{Sigma 1 2}-completeness}\label{secComplete}

Let $G \subseteq 2^\omega \times \mathcal{R}$ and $C \subseteq 2^\omega$ be as in Section \ref{secCalc}. By Proposition \ref{sigma12}, $C$ is $\mathbf{\Sigma}^1_2$. We present sufficient conditions for $C$ to be $\mathbf{\Sigma}^1_2$-complete.

Given any map $m: \mathcal{AR} \to \omega$, we define $\ell: \mathcal{AR} \to [\omega]^{< \omega}$ and $\overline{\ell}: \mathcal{R} \to [\omega]^{\leq \omega}$ by
$$\ell(s) := \set{m(t) \mid \emptyset \neq t \sqsubseteq s}$$
and
$$\overline{\ell}(A) := \set{m(t) \mid \emptyset \neq t \sqsubseteq A} = \bigcup_{n < \omega} \ell(r_n(A)).$$

\begin{definition} \label{properties}
    A topological Ramsey space $(\mathcal{R}, \leq, r)$ is {\it well-indexed} if there is a map $m: \mathcal{AR} \to \omega$ satisfying the following properties:
    \begin{enumerate}
        \item (Monotonicity) For all $s \sqsubseteq t$, $m(s) \leq m(t)$.\label{p1}
        \item (Unboundedness) For all $A$, $\overline{\ell}(A) \in [\omega]^\omega$.\label{p2}
        \item (Compatibility with $\leq$) There exists $A^*$ such that $\overline{\ell}(A^*)$ is maximal (i.e., $\overline{\ell}(A) \subseteq \overline{\ell}(A^*)$ for all $A$) and for all $B \leq A \leq A^*$, we have $\overline{\ell}(B) \subseteq \overline{\ell}(A)$.\label{p3}
        \item (Selection) For any $s \sqsubseteq A$ and $B' \in [\ell(s), \overline{\ell}(A)]$ with $\abs{\overline{\ell}(A) \setminus B'} = 1$, there is $B \in [s, A]$ such that $\overline{\ell}(B) = B'$. \label{p4}
    \end{enumerate}
\end{definition}

\begin{remark}\label{remProp4}
    Observe that in the presence of monotonicity, unboundedness, and compatibility, the selection property is equivalent to the following statement: 
    \begin{center}
        For any $s \sqsubseteq A$ and $B' \in [\ell(s), \overline{\ell}(A)]$, there is $B \in [s, A]$ such that $\overline{\ell}(B) = B'$.
    \end{center}
    Indeed, if $\overline{\ell}(A) \setminus B' = \set{p_i \mid i < \omega}$ with $p_0 < p_1 < \ldots$, we can construct $A \geq B_0 \geq B_1 \geq \ldots$ and $s \sqsubseteq s_0 \sqsubseteq s_1 \sqsubseteq \ldots$ such that $\overline{\ell}(B_i) = \overline{\ell}(A) \setminus \set{p_j \mid j \leq i}$, $s_i \sqsubseteq B_i$, and $\ell(s_i) = \overline{\ell}(B_i) \cap p_{i + 1}$. Then we can see that the limit $B := \lim_i B_i \in [s, A]$ exists and $\overline{\ell}(B) = B'$.
\end{remark}

Suppose $m: \mathcal{AR} \to \omega$ well-indexes $\mathcal{R}$ and $A^*$ witnesses that $m$ satisfies compatibility. Note that monotonicity implies $\ell(s) \sqsubseteq \ell(t)$ when $s \sqsubseteq t$, and $\ell(s) \sqsubseteq \overline{\ell}(A)$ when $s \sqsubseteq A$. In addition, we have $\ell[\mathcal{AR}] = [\overline{\ell}(A^*)]^{< \omega}$ and $\overline{\ell}[\mathcal{R}] = [\overline{\ell}(A^*)]^{\omega}$; compatibility and selection are key to this observation.

The following proposition, along with Proposition \ref{sigma12}, completes the proof of Theorem \ref{mainThm}.

\begin{proposition}\label{sigma12complete}
    Suppose $\mathcal{R}$ is well-indexed. Then $C$ is $\mathbf{\Sigma}^1_2$-hard.
\end{proposition}

\begin{proof}
Let $H \subseteq 2^\omega \times [\omega]^\omega$ be a universal $G_\delta$ set constructed in the same way as $G$. By \cite[Theorem 1]{Sabok}, the set $C' := \set{x \in 2^\omega \mid H_x \text{ is Ramsey positive}}$ is $\mathbf{\Sigma}^1_2$-complete. Thus, it suffices to construct a continuous reduction from $C'$ to $C$.

Let $m: \mathcal{AR} \to \omega$ well-index $\mathcal{R}$. We may assume that for $A^*$ witnessing compatibility, $\overline{\ell}(A^*) = \omega$: if not, let $g: \overline{\ell}(A^*) \to \omega$ be an order-preserving bijection, and consider $m': \mathcal{AR} \to \omega$ defined by $m'(\emptyset) := 0$ and $m'(s) := g(m(s))$ if $s \neq \emptyset$. Then $m'$ well-indexes $\mathcal{R}$, and our original $A^*$ witnesses compatibility with $\overline{\ell'}(A^*) = \omega$.

We define a map $f: \mathcal{P}([\omega]^{< \omega}) \to \mathcal{P}(\mathcal{AR})$ as follows. Given $X \subseteq [\omega]^{< \omega}$, define $f(X) \subseteq \mathcal{AR}$ by
$$s \in f(X) \iff \ell(s) \in X$$
for all $s \in \mathcal{AR}$. Note that $X$ is a tree with respect to $\sqsubseteq$ iff $f(X)$ is a tree with respect to $\sqsubseteq$, using monotonicity and the surjectivity of $\ell$. Moreover, in the case that $X$ and $f(X)$ are trees, the definition of $\overline{\ell}$, monotonicity, and unboundedness yield
\begin{equation}\label{branches}
    A \in [f(X)] \iff \overline{\ell}(A) \in [X]\tag{$*$}
\end{equation}
for all $A \in \mathcal{R}$.

Define the map $\varphi: 2^\omega \to 2^\omega$ such that if $x \in 2^\omega$ codes the sequence $\langle X_n \subseteq [\omega]^{< \omega} \mid n < \omega \rangle$, then $\varphi(x)$ codes the sequence $\langle f(X_n) \subseteq \mathcal{AR} \mid n < \omega \rangle$. Note that $\varphi$ is continuous. We claim that $\varphi$ is a reduction from $C'$ to $C$.

First, suppose $x \in C'$ and $x$ codes $\langle X_n \subseteq [\omega]^{< \omega} \mid n < \omega \rangle$. Since $x \in C'$, we must have that each $X_n$ is a tree and $[\omega]^\omega \setminus \bigcup_{n < \omega} [X_n]$ is Ramsey positive. Then each $f(X_n)$ is a tree, so
$$G_{\varphi(x)} = \mathcal{R} \setminus \bigcup_{n < \omega} [f(X_n)].$$
Fix some $s' \sqsubseteq A'$ such that $[s', A'] \cap [X_n] = \emptyset$ for every $n$. By compatibility and selection, we can find $s \sqsubseteq A \leq A^*$ such that $\ell(s) = s'$ and $\overline{\ell}(A) = A'$. To prove that $G_{\varphi(x)}$ is Ramsey positive, it suffices to show $[s, A] \cap [f(X_n)] = \emptyset$ for every $n$. Suppose for some $n < \omega$, we can find $B \in [s, A] \cap [f(X_n)]$. Let $B' := \overline{\ell}(B)$. Note that $B' \in [X_n]$ by (\ref{branches}). We have $s' \sqsubseteq B'$ and $B' \subseteq A'$ by monotonicity and compatibility, so $B' \in [s', A']$. But this contradicts that  $[s', A'] \cap [X_n] = \emptyset$. So we conclude $[s, A] \cap [f(X_n)] = \emptyset$ for each $n$, hence $G_{\varphi(x)}$ is Ramsey positive and $\varphi(x) \in C$.

Now suppose we have $x \in 2^\omega$ such that $\varphi(x) \in C$ and $x$ codes $\langle X_n \subseteq [\omega]^{< \omega} \mid n < \omega \rangle$. Since $\varphi(x) \in C$, every $f(X_n)$ is a tree and $\mathcal{R} \setminus \bigcup_{n < \omega} [f(X_n)]$ is Ramsey positive. So we have that every $X_n$ is a tree and
$$H_x = [\omega]^\omega \setminus \bigcup_{n < \omega} [X_n].$$
Fix $s \sqsubseteq A$ such that $[s, A] \cap [f(X_n)] = \emptyset$ for all $n < \omega$. Let $s' := \ell(s)$ and $A' := \overline{\ell}(A)$. Then $s' \sqsubseteq A'$ by monotonicity. We claim $[s', A'] \cap [X_n] = \emptyset$ for all $n < \omega$. Suppose for some $n < \omega$, we can find $B' \in [s', A'] \cap [X_n]$. By selection, there is $B \in [s, A]$ such that $\overline{\ell}(B) = B'$. By (\ref{branches}), we have $B \in [f(X_n)]$ since $B' \in [X_n]$. But this contradicts that $[s, A] \cap [f(X_n)] = \emptyset$. We conclude that $H_x$ is Ramsey positive, thus $x \in C'$.

So $\varphi$ is a continuous reduction from $C'$ to $C$. Therefore, $C$ is $\mathbf{\Sigma}^1_2$-hard.
\end{proof}

% ----------------- EXAMPLES ---------------------------------
\section{Examples}\label{secExamples}
We present a few examples of topological Ramsey spaces where Theorem \ref{mainThm} shows that the set of codes for $G_\delta$ Ramsey positive sets is $\mathbf{\Sigma}^1_2$-complete. For each, we exhibit a map $m$ that well-indexes the space.

\subsection{Ellentuck space}
% ----------- Ellentuck space ----------------------
The proof of Proposition \ref{sigma12complete} relies on Sabok's result that the set of codes for $G_\delta$ Ramsey positive sets of $[\omega]^\omega$ is $\mathbf{\Sigma}^1_2$-complete, so showing that $[\omega]^\omega$ is well-indexed is superfluous. We nevertheless provide such a map for the sake of illustration.

Define $m: [\omega]^{< \omega} \to \omega$ by $m(\emptyset) := 0$ and $m(s) := \max(s)$ for $s \neq \emptyset$. Then $\ell(s) = s$ and $\overline{\ell}(A) = A$ for all $s \in [\omega]^{< \omega}$ and $A \in [\omega]^\omega$. It is clear that monotonicity and unboundedness hold. Letting $A^* = \omega$, we see that compatibility holds as well. Finally, selection holds since we can pick $B = B'$. So $m$ well-indexes $[\omega]^\omega$.

\subsection{Strong subtrees of \texorpdfstring{$2^{< \omega}$}{2 <omega}}
% -------------- Strong trees ----------------------
Let $2^{< \omega}$ denote the complete binary tree. We say that $A \subseteq 2^{< \omega}$ is a {\it strong subtree} of $2^{< \omega}$ if there exists a set of levels $I \subseteq \omega$ such that
\begin{enumerate}
    \item $A \subseteq \bigcup_{i \in I} 2^i$;
    \item if $\set{i_k \mid k < \abs{I}}$ is the increasing enumeration of $I$, $\abs{A \cap 2^{i_k}} = 2^k$ for each $k < \abs{I}$;
    \item if $i \in I$, $i'= \min \set{n \in I \mid n > i}$, $a \in A \cap 2^i$, and $j \in \set{0, 1}$, then there is exactly one $b \in A \cap 2^{i'}$ extending $a^\frown (j)$.
\end{enumerate}
Let $\mathcal{S}_\infty(2^{< \omega})$ denote the set of all infinite strong subtrees of $2^{< \omega}$. If $\set{i_k \mid k < \omega}$ is the increasing enumeration of the levels of $A$, define $r_n(A) = A \cap \bigcup_{k < n} 2^{i_k}$ for each $n < \omega$. Then $(\mathcal{S}_\infty(2^{< \omega}), \subseteq, r)$ is a topological Ramsey space; see \cite{Milliken} and \cite[Section 6.1]{Ramsey} for details.

Let $\mathcal{S}_{< \infty}(2^{< \omega}) = \mathcal{A} \mathcal{S}_\infty(2^{< \omega})$ denote the set of finite strong subtrees of $2^{< \omega}$. Define $m: \mathcal{S}_{< \infty}(2^{< \omega}) \to \omega$ by
\begin{equation*}
    m(s) := \begin{cases}
        0, & \quad \text{if } s = \emptyset \\
        \max \set{n \mid s \cap 2^n \neq \emptyset}, & \quad \text{otherwise.}
    \end{cases}
\end{equation*}
Observe that for any $s \in \mathcal{S}_{< \infty}(2^{< \omega})$ and $A \in \mathcal{S}_{\infty}(2^{< \omega})$, $\ell(s) = \set{n \mid s \cap 2^n \neq \emptyset}$ and $\overline{\ell}(A) = \set{n \mid A \cap 2^n \neq \emptyset}$ are the corresponding sets of levels. Monotonicity and unboundedness are clear, and compatibility holds with $A^* = 2^{< \omega}$. Selection also holds: intuitively, finding an appropriate $B \in [s, A]$ given $B' = \overline{\ell}(A) \setminus \set{k} \in [\ell(s), \overline{\ell}(A)]$ amounts to thinning the nodes in $A$ at levels above $k$.

Thus, $\mathcal{S}_{\infty}(2^{< \omega})$ is well-indexed, and Theorem \ref{mainThm} implies that the set of codes for $G_\delta$ Ramsey positive subsets of $\mathcal{S}_{\infty}(2^{< \omega})$ is $\mathbf{\Sigma}^1_2$-complete.

\subsection{Infinite sequences of words with variables}
% ------------- Words -----------------------
Let $L = \bigcup_{n < \omega} L_n$ be a set, where each $L_n$ is finite and $L_n \subseteq L_{n + 1}$. Let $v \notin L$. Define $W_{Lv}$ to be the set of all words---non-empty finite strings of elements from $L \cup \set{v}$---in which $v$ appears. Define
$$W_{Lv}^{[\infty]} := \set{\langle a_n \rangle_{n < \omega} \mid \forall n < \omega~\big(a_n \in W_{Lv} \wedge \abs{a_n} > \sum_{i < n} \abs{a_i} \big)}.$$
Given $A \in W_{Lv}^{[\infty]}$, let
$$[A]_{Lv} := \set{a_{n_0}[\lambda_0]^\frown \ldots^\frown a_{n_k}[\lambda_k] \in W_{Lv} \mid n_0 < \ldots < n_k \wedge \forall 0 \leq i \leq k~\lambda_i \in L_{n_i} \cup \set{v}},$$
where $a[\lambda]$ denotes the string obtained by replacing every instance of $v$ in $a$ with $\lambda$. If $a = a_{n_0}[\lambda_0]^\frown \ldots^\frown a_{n_k}[\lambda_k] \in [A]_{Lv}$, define
$$\text{supp}_A(a) := \set{n_0, \ldots, n_k}.$$
Note that the condition on the lengths of the words in $A$ guarantees that $\text{supp}_A(a)$ is well-defined.

For $A = \langle a_n \rangle_{n < \omega}, B = \langle b_n \rangle_{n < \omega} \in W_{Lv}^{[\infty]}$, define $A \leq B$ iff 
\begin{enumerate}
    \item for all $n < \omega$, $a_n \in [B]_{Lv}$;
    \item for all $k < n$, $\max(\text{supp}_B(a_k)) < \min(\text{supp}_B(a_n))$.
\end{enumerate}
For $k < \omega$ and $A = \langle a_n \rangle_{n < \omega} \in W_{Lv}^{[\infty]}$, define $r_k(A) := \langle a_n \rangle_{n < k}$. Then $(W_{Lv}^{[\infty]}, \leq, r)$ is a topological Ramsey space; see \cite{Carlson} and \cite[Section 5.3]{Ramsey} for details.

Let $W_{Lv}^{[< \infty]} := \mathcal{A}W_{Lv}^{[\infty]}$. Define $m: W_{Lv}^{[< \infty]} \to \omega$ by $m(\emptyset) := 0$ and $m(\langle a_i \rangle_{i < n}) := \lfloor \log_2 \abs{a_{n - 1}} \rfloor$ if $n \geq 1$. It is clear that monotonicity and unboundedness are satisfied. We can see that compatibility holds if we let $A^* = \langle a_n^* \rangle_{n < \omega}$, where $a_n^*$ is the word $vv \ldots v$ consisting of $v$ exactly $2^n$ times. We remark that although we have $\overline{\ell}(B) \subseteq \overline{\ell}(A)$ when $B \leq A \leq A^*$, this is not true for all $B \leq A$. Finally, to show selection, suppose $s \sqsubseteq A$ and $B' = [\ell(s), \overline{\ell}(A)]$ with $\overline{\ell}(A) \setminus B' = \set{k}$. If $A = \langle a_n \rangle_{n < \omega}$, we can form $B \in [s, A]$ with $\overline{\ell}(B) = B'$ by setting $B = \langle a_n \mid n < \omega, \lfloor \log_2 \abs{a_n} \rfloor \neq k \rangle$.

We conclude that $W_{Lv}^{[< \infty]}$ is well-indexed, hence the set of codes for $G_\delta$ Ramsey positive subsets of $W_{Lv}^{[< \infty]}$ is $\mathbf{\Sigma}^1_2$-complete by Theorem \ref{mainThm}.

\subsection{High-dimensional Ellentuck spaces}\label{Ek}
% --------------- High dimension ----------------------
For $k \geq 2$, we define the topological Ramsey space $\mathcal{E}_k$ as in \cite{Dobrinen}. Let $\omega^{\not\downarrow \leq k}$ denote the set of all non-decreasing sequences of natural numbers of length at most $k$, that is,
$$\omega^{\not\downarrow \leq k} = \set{\langle u_0, u_1, \ldots, u_{p - 1} \rangle \mid 0 \leq p \leq k \wedge u_0 \leq u_1 \leq \ldots \leq u_{p - 1}}.$$
Let $\omega^{\not\downarrow k}$ denote the set of all non-decreasing sequences of natural numbers of length exactly $k$.

Define the well-order $\prec$ on $\omega^{\not\downarrow \leq k}$ such that
\begin{enumerate}
    \item $\langle \rangle$ is the $\prec$-minimum;
    \item for $\langle u_0, \ldots, u_{p - 1} \rangle, \langle v_0, \ldots, v_{q - 1} \rangle \in \omega^{\not\downarrow \leq k}$ with $p, q > 0$, $\langle u_0, \ldots, u_{p - 1} \rangle \prec \langle v_0, \ldots, v_{q - 1} \rangle$ iff either
    \begin{enumerate}
        \item $u_{p - 1} < v_{q - 1}$, or
        \item $u_{p - 1} = v_{q - 1}$ and $\langle u_0, \ldots, u_{p - 1} \rangle <_{\text{lex}} \langle v_0, \ldots, v_{q - 1} \rangle$.
    \end{enumerate}
\end{enumerate}
Then with respect to $\prec$, $\omega^{\not\downarrow \leq k}$ has order type $\omega$. For $a < \omega$, let $\vec{j}_a$ denote the $a^{\text{th}}$ element of $(\omega^{\not\downarrow \leq k}, \prec)$. Given any $\vec{u} \in \omega^{\not\downarrow \leq k}$, define $a_{\vec{u}}$ to be the unique $a < \omega$ such that $\vec{u} = \vec{j}_a$. For $b < \omega$, let $\vec{i}_b$ denote the $b^{\text{th}}$ element of $\omega^{\not\downarrow k}$ with respect to the order $\prec$, inherited from $\omega^{\not\downarrow \leq k}$.

We define a function $\widehat{W_k}: \omega^{\not\downarrow \leq k} \to [\omega]^{\leq k}$ as follows. For $\vec{u} \in \omega^{\not\downarrow \leq k}$ with $\abs{\vec{u}} = p$, define $\widehat{W_k}(\vec{u}) := \set{a_{\vec{u} \restriction q} \mid 1 \leq q \leq p}$. Note that
$$a_{\vec{u} \restriction 1} < a_{\vec{u} \restriction 2} < \ldots < a_{\vec{u} \restriction p} = a_{\vec{u}},$$
so $\widehat{W_k}(\vec{u}) \in [\omega]^p$. Let $W_k := \widehat{W_k} \restriction \omega^{\not\downarrow k}$. Define
$$\mathbb{W}_k := W_k[\omega^{\not\downarrow k}] \subseteq [\omega]^k$$
and
$$\widehat{\mathbb{W}_k} := \widehat{W_k}[\omega^{\not\downarrow \leq k}] \subseteq [\omega]^{\leq k}.$$
Observe that $\widehat{\mathbb{W}_k}$ is the tree formed from all initial segments of elements of $\mathbb{W}_k$.

A function $\widehat{A}: \omega^{\not\downarrow \leq k} \to \widehat{\mathbb{W}_k}$ is an {\it $\mathcal{E}_k$-tree} if it satisfies the following conditions:
\begin{enumerate}
    \item for all $a < \omega$, $\widehat{A}(\vec{j}_a) \in [\omega]^{\abs{\vec{j}_a}}$; \label{Ek_p1}
    \item for $1 \leq a < \omega$, $\max(\widehat{A}(\vec{j}_a)) < \max(\widehat{A}(\vec{j}_{a + 1}))$; \label{Ek_p2}
    \item for $a, b < \omega$, $\widehat{A}(\vec{j}_a) \sqsubseteq \widehat{A}(\vec{j}_b)$ iff $\vec{j}_a \sqsubseteq \vec{j}_b$. \label{Ek_p3}
\end{enumerate}
Given an $\mathcal{E}_k$-tree $\widehat{A}$, define
$$[\widehat{A}] := \widehat{A} \cap (\omega^{\not\downarrow k} \times \mathbb{W}_k) = \widehat{A} \restriction \omega^{\not\downarrow k}.$$
Define
$$\mathcal{E}_k := \set{[\widehat{A}] \mid \widehat{A} \text{ is an } \mathcal{E}_k \text{-tree}}.$$
Note that condition (\ref{Ek_p2}) guarantees that every $A \in \mathcal{E}_k$ is uniquely determined by $A[\omega^{\not\downarrow k}] = \set{A(\vec{i}_n) \mid n < \omega}$. Furthermore, there is a unique $\mathcal{E}_k$-tree $\widehat{A}$ for which $[\widehat{A}] = A$ by conditions (\ref{Ek_p1}) and (\ref{Ek_p3}). For $A, B \in \mathcal{E}_k$, we define $A \leq B$ iff $A[\omega^{\not\downarrow k}] \subseteq B[\omega^{\not\downarrow k}].$ For $n < \omega$ and $A \in \mathcal{E}_k$, let $r_n(A) := A \restriction \set{\vec{i}_p \mid p < n}$. By \cite[Theorem 3.17]{Dobrinen}, $(\mathcal{E}_k, \leq, r)$ is a topological Ramsey space.

Where convenient, we will write $\widehat{A}(u_0, u_1, \ldots, u_{p - 1})$ instead of $\widehat{A}(\langle u_0, u_1, \ldots, u_{p - 1} \rangle)$ below. Define $m: \mathcal{A}\mathcal{E}_k \to \omega$ as follows. Let $m(\emptyset) := 0$. Given $s$ with domain $\set{\vec{i}_n \mid n < p}$ for some $p > 0$, let $M_s := \max \set{M \mid \exists n < p~\langle M \rangle \sqsubseteq \vec{i}_n}$. Note that there must be $n_s < p$ such that $\vec{i}_{n_s} = \langle M_s, M_s, \ldots, M_s \rangle$. Define $m(s) := \min(s(\vec{i}_{n_s}))$.

We check that $m$ well-indexes $\mathcal{E}_k$. To see monotonicity, note that if $\emptyset \neq s \sqsubseteq t$, then we have $M_s \leq M_t$, $\langle M_s \rangle \preceq \langle M_t \rangle$, and $s(\vec{i}_{n_s}) = t(\vec{i}_{n_s})$. Fix any $A \in \mathcal{E}_k$ with $t \sqsubseteq A$. Then
$$m(s) = \min(t(\vec{i}_{n_s})) = \max(\widehat{A}(M_s)) \leq \max(\widehat{A}(M_t)) = \min(t(\vec{i}_{n_t})) = m(t),$$
so monotonicity holds.

It is not difficult to see that $\overline{\ell}(A) = \set{\max(\widehat{A}(n)) \mid n < \omega}$ for all $A \in \mathcal{E}_k$. Since $\max(\widehat{A}(0)) < \max(\widehat{A}(1)) < \ldots$ by condition (\ref{Ek_p2}), we have $\overline{\ell}(A) \in [\omega]^\omega$, hence unboundedness is satisfied.

We next verify compatibility. If we set $A^* = W_k$, then $m(A^*)$ is maximal. It suffices to verify that $B \leq A \leq A^*$ implies $\overline{\ell}(B) \subseteq \overline{\ell}(A)$. Note that
$$\overline{\ell}(B) = \set{\max(\widehat{B}(n)) \mid n < \omega} \subseteq \set{\max(\widehat{A}(n)) \mid n < \omega} = \overline{\ell}(A)$$
whenever $B \leq A$, so $m$ satisfies compatibility.

Finally, we show that $m$ satisfies selection. Consider $s \sqsubseteq A$ and $B' \in [\ell(s), \overline{\ell}(A)]$ with $\overline{\ell}(A) \setminus B' = \set{x}$. Fix $u < \omega$ such that $\widehat{A}(u) = \set{x}$. Define $\widehat{B}: \omega^{\not\downarrow \leq k} \to \widehat{\mathbb{W}_k}$ as follows. Consider any $\langle u_0, u_1, \ldots, u_{p - 1} \rangle \in \omega^{\not\downarrow \leq k}$. For $q < p$, define
$$v_q := \begin{cases}
    u_q, & \quad \text{if } u_q \leq u \text{ and } u_0 \neq u \\
    u_q + 1, & \quad \text{otherwise}.
\end{cases}$$
Let $\widehat{B}(u_0, u_1, \ldots, u_{p - 1}) := \widehat{A}(v_0, v_1, \ldots, v_{p - 1})$. We can verify that $\widehat{B}$ is an $\mathcal{E}_k$-tree, so we have $B := [\widehat{B}] \in \mathcal{E}_k$. By definition of $B$, we have $B \leq A$. Note that since $x > \max(\ell(s))$ and $B(\vec{i}) = A(\vec{i})$ for all $\vec{i} \in \omega^{\not\downarrow k}$ with $\vec{i} \prec \langle u \rangle$, we must have $s \sqsubseteq B$. So $B \in [s, A]$. Since $\widehat{B}(v) = \widehat{A}(v)$ for all $v < u$ and $\widehat{B}(v) = \widehat{A}(v + 1)$ for $v \geq u$, we have
$$\overline{\ell}(B) = \overline{\ell}(A) \setminus \set{\max(\widehat{A}(u))} = \overline{\ell}(A) \setminus \set{x} = B'.$$

Thus, $\mathcal{E}_k$ is well-indexed, hence Theorem \ref{mainThm} applies.

\subsection{Other examples}
We provide several other examples to which Theorem \ref{mainThm} applies. We exhibit a map $m$ well-indexing each space, using the notation from the indicated reference. The verification that these functions satisfy the required properties is omitted.

\begin{enumerate}
    \item $\text{FIN}_k^{[\infty]}$ as defined in \cite[Section 5.2]{Ramsey}: Define $m: \text{FIN}_k^{[< \infty]} \to \omega$ by $m(\emptyset) := 0$ and $m((p_i)_{i < n}) := \min \set{j \mid p_{n - 1}(j) = k}$ for $n \geq 1$.
    \item $\text{FIN}_{Lv}^{[\infty]}$ as defined in \cite[Section 5.3]{Ramsey}: Define $m: \text{FIN}_{Lv}^{[< \infty]} \to \omega$ by $m(\emptyset) := 0$ and $m((x_i)_{i < n}) := \min \set{j \mid x_{n - 1}(j) = v}$ for $n \geq 1$.
    \item $\mathcal{E}_\infty$ as defined in \cite[Section 5.6]{Ramsey}: Define $m: \mathcal{AE}_\infty \to \omega$ by $m(\emptyset) := 0$ and $m(r_n(E)) := p_n(E)$ if $n \geq 1$.
    \item $\mathcal{M}_\infty$ as defined in \cite[Section 5.7]{Ramsey}: Define $m: \mathcal{AM}_\infty \to \omega$ by $m(\emptyset) := 0$ and $m(r_n(A)) := p_{n - 1}(A)$ if $n \geq 1$.
    \item $\mathcal{R}_1$ as defined in \cite{DobrinenTodorcevic}: Define $m: \mathcal{AR} \to \omega$ by $m(\emptyset) := 0$ and $m(a) := k_{n - 1}$ for $a \in \mathcal{AR}_n$ with $n \geq 1$ and $a(i) \subseteq \mathbb{T}(k_i)$ for each $i < n$.
\end{enumerate}

There are several generalizations of spaces mentioned above that we have not been able to prove are well-indexed. In \cite{DobrinenInf}, Dobrinen defined spaces $\mathcal{E}_B$ for uniform barriers on $\omega$ as an extension of the spaces $\mathcal{E}_k$ from \cite{Dobrinen}. We suspect that a map similar to the $m$ defined in Section \ref{Ek} may well-index $\mathcal{E}_B$; however, we have not been able to prove it. Likewise, we do not know if we can apply Theorem \ref{mainThm} to the spaces $\mathcal{R}_\alpha$, $\alpha < \omega_1$, defined in \cite{DobrinenTodorcevic2} as an extension of the space $\mathcal{R}_1$ defined in \cite{DobrinenTodorcevic}.

% ------------------------------- QUESTIONS ----------------------------
\section{Further questions}
For all the topological Ramsey spaces we have considered so far, we either have a map that well-indexes the space or have a plausible candidate map. We ask the following related questions.

\begin{problem}
    Are all topological Ramsey spaces well-indexed?
\end{problem}

\begin{problem}
    Is there a topological Ramsey space whose set of codes for $G_\delta$ Ramsey positive sets is not $\mathbf{\Sigma}^1_2$-complete?
\end{problem}

In particular, we are interested to know whether a map $m: \mathcal{AR} \to \omega$ well-indexing $\mathcal{R}$ can be produced solely using the assumptions A.1-A.4 from \cite[Section 5.1]{Ramsey}.

\bibliographystyle{alpha}
\bibliography{bibliography}

\end{document}